\documentclass{conm-p-l}

\usepackage[centering]{geometry}
\usepackage[utf8]{inputenc}
\usepackage[english]{babel}
\usepackage{mathtools}
\usepackage{graphicx}
\usepackage{amsmath,amssymb,amsthm}
\usepackage{comment}
\usepackage{enumitem}
\usepackage{leftidx}
\usepackage{tikz}
\usetikzlibrary{cd}

\usepackage[
colorlinks=true,
urlcolor=purple,
linkcolor=purple!87!black,
pdfborder={0 0 0}
]{hyperref}

\renewcommand{\phi}{\varphi}

\newcommand{\findim}{\mathrm{findim}\,}
\newcommand{\Findim}{\mathrm{Findim}\,}
\newcommand{\dell}{\mathrm{dell}\,}
\newcommand{\phidim}{\varphi\dim}
\newcommand{\psidim}{\psi\dim}
\renewcommand{\dell}{\mathrm{dell}\,}

\newcommand{\pd}{\mathrm{pd}\,}
\newcommand{\gldim}{\mathrm{gldim}\,}

\newcommand{\op}{\mathrm{op}}

\newcommand{\N}{\mathbb{N}}
\renewcommand{\k}{\mathbb{K}}

\renewcommand{\mod}{\mathrm{mod}\,}
\newcommand{\Mod}{\mathrm{Mod}\,}

\newcommand{\dirsum}{\xhookrightarrow{\!\oplus}}
\newcommand{\into}{\xhookrightarrow{}}

\newcommand{\agemo}{\raisebox{\depth}{\scalebox{1}[-1]{\(\Omega\)}}}
\newcommand{\Tr}{\mathrm{Tr}\,}

\newcommand{\depth}{\mathrm{depth}\,}

\newcommand{\kdell}{k\text{-}\mathrm{dell}\,}
\newcommand{\ddell}{\mathrm{ddell}\,}

\newcommand{\subddell}{\mathrm{sub}\text{-}\mathrm{ddell}\,}

\newtheorem{theorem}{Theorem}[section]
\newtheorem{lemma}[theorem]{Lemma}
\newtheorem{proposition}[theorem]{Proposition}
\newtheorem{corollary}[theorem]{Corollary}

\theoremstyle{definition}
\newtheorem{definition}[theorem]{Definition}
\newtheorem{example}[theorem]{Example}

\theoremstyle{remark}
\newtheorem{remark}[theorem]{Remark}
\newtheorem{question}[theorem]{Question}

\numberwithin{equation}{section}

\begin{document}

\title{Homological Invariants of Left and Right Serial Quiver Algebras}

\author{Ruoyu Guo}
\address{415 South Street, Department of Mathematics, Brandeis University, Waltham, MA 02453}
\curraddr{}
\email{rguo@brandeis.edu}

\date{}

\maketitle

\begin{abstract}
We investigate the relationship between the delooping level ($\dell\!$) and the finitistic dimension of left and right serial quiver algebras. These 2-syzygy finite algebras have finite delooping level, and it can be calculated with an easy and finite algorithm. When the algebra is right serial, its right finitistic dimension is equal to its left delooping level. When the algebra is left serial, the above equality only holds under certain conditions. We provide examples to demonstrate this and include discussions on the sub-derived ($\subddell\!$) and derived delooping level ($\ddell\!)$. Both $\subddell\!$ and $\ddell\!$ are improvements of the delooping level. We motivate their definitions and showcase how they can behave better than the delooping level in certain situations throughout the paper.
\end{abstract}

\section{Introduction and Definitions}

We present some new results on the finitistic dimension conjecture over finite dimensional algebras. This important homological conjecture in representation theory is the sufficient condition for numerous other conjectures including the Nakayama conjecture, Gorenstein symmetry conjecture, and Anslander-Reiten conjecture, to name a few. The representational approaches covered in this paper rely on studying the properties of homological invariants and creating new ones that are upper or lower bounds of the finitistic dimension. They are the delooping level \cite{gelinas2022}, sub-derived delooping level, and the derived delooping level \cite{guo2025derived}. We study their behavior in monomial algebras and specialize to left and right serial quiver algebras. As subclasses of monomial algebras, left and right serial quiver algebras enjoy a lot of good properties such as 2-syzygy finiteness and having a tractable syzygy structure. We can also calculate their little and big finitistic dimensions with the margin of error at most one \cite{huisgen1991predicting}. From the class of monomial algebras we find one of the first examples where the big and the little finitistic dimensions differ \cite{huisgen1992homological}. Various bounds of the finitistic dimension of monomial algebras are studied using different methods \cite{green1985projective, GKK1991, igusa1990syzygy, shi2003finitistic}. Despite the large amount of work on monomial algebras, their finitistic dimensions still need to be calculated on a case-by-case basis. Our main result shows we can calculate the right finitistic dimension through the left delooping level if the quiver algebra is right serial or left serial with an additional condition. This adds to the list of algebras whose big finitistic dimension is described by the delooping level of the opposite algebra. The main results are a series of equalities when $\Lambda$ is a right serial algebra or left serial under an additional condition
\[
\findim\Lambda = \Findim\Lambda = \dell\Lambda^{\op} = \ddell\Lambda^{\op} < \infty,
\]
where all the homological dimensions use \textbf{right} modules by default. When the equality does not hold in the left serial case, we provide a representation-finite algebra in Example \ref{ex:right serial} such that $\findim\Lambda = \Findim\Lambda = \ddell\Lambda^{\op} < \dell\Lambda$. We also recover known results on Nakayama algebras and monomial algebras of acyclic quivers using our proposed method along the way. While $\Findim\Lambda = \ddell\Lambda^{\op}$ holds in several examples in which $\Findim\Lambda<\dell\Lambda^{\op}$ such as in \cite[Example 3.8]{guo2025derived} and Example \ref{ex:right serial}, $\Findim\Lambda=\ddell\Lambda^{\op}$ may not be true in general considering the example \cite[Example 4.22]{barrios2024delooping}. There is still much to study to what extent the derived delooping level can describe the big finitistic dimension.

For the rest of the paper, let $\Lambda$ be a finite dimensional algebra over an algebraically closed field $\k$. Let $\mod\Lambda$ and $\Mod\Lambda$ be the category of finitely generated \textbf{right} $\Lambda$-modules and the category of all \textbf{right} $\Lambda$-modules. The (right) little and big finitistic dimension conjectures state that for a finite dimensional algebra $\Lambda$, the little finitistic dimension $\findim\Lambda$ and the big finitistic dimension $\Findim\Lambda$, respectively defined as
\[
\findim\Lambda = \sup \{\pd M\mid M\in\mod\Lambda, \pd M<\infty\}
\]
and
\[
\Findim\Lambda = \sup \{\pd M\mid M\in\Mod\Lambda, \pd M<\infty\},
\]
are finite, where $\pd M$ is the projective dimension of $M$. Note that the original formulation of the finitistic dimension conjecture contains another statement $\findim\Lambda=\Findim\Lambda$. This is proved to be false in the case of monomial algebras \cite{huisgen1992homological} and is not the focus of this paper.

A quiver $Q=(Q_0, Q_1, s, t)$ is a directed graph with four pieces of information, where $Q_0$ is the vertex set, $Q_1$ is the arrow set, and $s,t:Q_1\to Q_0$ are the starting and terminal vertices of an arrow in $Q_1$. We frequently consider paths in the quiver, so for convenience, we extend the domain of $s$ and $t$ to include all paths in $Q$ in the natural way. For each quiver $Q$, we can associate with it a path algebra whose $\k$-basis is the set of all paths in $Q$ and multiplication is path concatenation. Introductions to quiver representations can be found in \cite{assem2006elements, schiffler2014quiver}. For a quiver algebra $\Lambda=\k Q/I$ where $I$ is an admissible ideal, we denote by $P_v$ and $S_v$ the indecomposable projective and simple modules whose top is supported on the vertex $v$, respectively.

Quiver algebras provide a wealth of examples for studying this conjecture, and they are very general in the finite dimensional algebra case in the following sense. Every finite dimensional algebra over a field $\k$ is Morita equivalent to a basic finite dimensional algebra, which is then isomorphic to some quiver algebra $\k Q/I$ subject to relations $I$ when $\k$ is algebraically closed. Since the finitistic dimension is invariant under Morita equivalence and field extensions \cite{jensen1982homological}, it suffices to assume $\k$ is algebraically closed and study the finitistic dimensions of quiver algebras.

The organization of the paper is as follows. In Section \ref{sec:variants of the delooping level}, we recall two invariants related to the delooping level called the sub-derived delooping level and derived delooping level, focusing on the motivation of their definition. The derived delooping level is especially better in terms of its properties and as an upper bound. In Section \ref{sec:Monomial algebras}, we introduce the technique that we use for proving the main theorems and recover some known results along the way. Section \ref{sec:left and right serial path algebras} contains the main result on left and right serial algebras and an illuminating representation-finite algebra example.

\textbf{Acknowledgments.} The author is grateful for his advisor Kiyoshi Igusa for helpful discussions on the paper.

\section{Variants of the Delooping Level}
\label{sec:variants of the delooping level}

We first recall the definition of the delooping level $\dell\Lambda$ for an algebra $\Lambda$. For two modules $M$ and $N$, let $M \dirsum N$ mean $M$ is a \textbf{direct summand} of $N$. Since we do not need to consider projective summands when calculating the projective dimension, modules considered hereafter have their projective summands omitted unless stated otherwise.
\begin{definition}[\cite{gelinas2022}]
Let $\agemo=\Tr\Omega\Tr$ be the left adjoint of the syzygy functor $\Omega$ in $\underline{\mod\!}\,\Lambda$. Define the delooping level of a $\Lambda$-module $M$ as
\[
\dell M = \inf \{n\in\N\mid \Omega^n M \text{ is a direct summand of } \Omega^{n+1} N \text{ for some module } N\},
\]
and we can show that
\begin{equation}
\label{eq:def of dell}
\dell M = \inf \{n\in\N\mid \Omega^n M \dirsum \Omega^{n+1}\agemo^{n+1}\Omega^n M\}.
\end{equation}
Define the \textbf{delooping level} of an algebra $\Lambda$ as
\[
\dell\Lambda = \sup\{\dell S\mid \text{$S$ is a simple $\Lambda$-module}\}.
\]

If $M$ is isomorphic to the direct summand of $\Omega^k(N)$ for some $k$ and some module $N$, then we say $M$ is a \textbf{$k$-syzygy} and that $M$ is \textbf{$k$-deloopable}. If $M \dirsum \Omega^i N_i$ for some $N_i$ and every $i\in\N$, then we say $M$ is \textbf{infinitely deloopable}.
\end{definition}
It is proved in \cite{gelinas2022} that $\Findim\Lambda\leq\dell\Lambda^{\op}$. The delooping level does not need to be finite a priori, but it is often small and in fact equal to the finitistic dimension in many cases. The cases where we know $\Findim\Lambda=\dell\Lambda^{\op}$ include
\begin{itemize}
\item algebras with finite global dimension
\item Gorenstein algebras \cite{gelinas2022}
\item Nakayama algebras \cite{ringel2021finitistic, sen2021delooping}
\item radical square zero algebras \cite{gelinas2021finitistic}
\item truncated path algebras \cite{barrios2024delooping}
\end{itemize}
and we add to this list right serial quiver algebras and a subclass of left serial quiver algebras in Section \ref{sec:left and right serial path algebras}. However, the equality is not true for monomial algebras in general. If $\Lambda$ is monomial, the difference $\dell\Lambda^{\op}-\Findim\Lambda$ is finite but can be arbitrarily large.

Despite many promising results about the delooping level, there is an example in \cite{kershaw2023} where $\dell\Lambda=\infty$ for a finite dimensional algebra $\Lambda$, where $\Lambda$ is a quiver algebra with only two vertices. Moreover, that same example shows that the set of finitely generated modules with finite delooping level is not closed under extension, submodules, or quotients. This means the delooping level does not necessarily behave well under exact sequences, so the delooping level of an arbitrary module has to be calculated or estimated independently using the definition \eqref{eq:def of dell} without notable properties. However, if we know the delooping level of a module is finite, \eqref{eq:def of dell} provides a finite algorithm for finding what it is.

The $\phi$-dimension and $\psi$-dimension \cite{igusa2005} are two other related invariants that have been widely-used for over two decades. They satisfy $\findim\Lambda\leq\phidim\Lambda\leq\psidim\Lambda$.  In contrast to the delooping level, the $\psi$-dimension and projective dimension interact very well under short exact sequences in the sense that if $0\to M_1 \to M_2 \to M_3\to 0$ is a short exact sequence and $\pd M_3<\infty$, then $\pd M_3\leq \psi(M_1\oplus M_2) + 1$. This property has allowed many applications for the $\psi$-dimension, one of the strongest of which shows $\findim\Lambda<\infty$ for $\Lambda$ with representation dimension no more than 3. Other applications include the Igusa-Todorov algebras \cite{wei2009finitistic} and the more general LIT algebras \cite{bravo2021generalised}. Due to the difference in delooping level's properties, these results cannot yet be replicated for the delooping level. For example, it is still an open question if $\Findim\Lambda<\infty$ for representation dimension 3 algebras.

The invariants in the rest of the section were introduced in \cite{guo2025derived}, but we hope to present their motivations more cohesively and pose future questions regarding their applications. In order to strengthen the properties of the delooping level, we first consider if there are any implications for submodules of a module with finite delooping level. This question leads to the definition of the sub-derived delooping level $\subddell\Lambda$ and a better upper bound for $\Findim\Lambda^{\op}$.

\begin{definition}[\cite{guo2025derived}]
\label{def:subddell}
The \textbf{sub-derived delooping levels} of a module $M$ and of an algebra $\Lambda$ are
\begin{itemize}
\item $\subddell M = \inf\{\dell N \mid M\into N\}$
\item $\subddell \Lambda = \sup\{\subddell S \mid S \text{ is a simple $\Lambda$-module} \}$
\end{itemize}
\end{definition}

\begin{theorem}
For any algebra $\Lambda$,
\label{thm:subddell}
\[
\Findim\Lambda^{\op} \leq \subddell\Lambda \leq \dell\Lambda.
\]
\end{theorem}

The sub-derived delooping level can be strictly better than the delooping level as an upper bound in Example \ref{ex:subddell} below and can also give no improvement in \cite[Example 3.8]{guo2025derived}, where $\infty=\subddell\Lambda=\dell\Lambda > \Findim\Lambda^{\op} = 1$. The next example is one of the smallest possible with $\dell\Lambda>\Findim\Lambda^{\op}=\subddell\Lambda$.

\begin{example}
\label{ex:subddell}
Let $\Lambda=\k Q/I$ be the quiver algebra of the following quiver $Q$

\begin{center}
\begin{tikzcd}[ampersand replacement=\&]
1 \arrow[r,shift left] \arrow[r,leftarrow,shift right] \arrow[loop left] \& 2 \arrow[loop right]
\end{tikzcd} 
\end{center}
with the following indecomposable projectives and injectives
\begingroup
\setlength\arraycolsep{1pt}
\begin{center}
Projectives: $\begin{matrix}
{} & 1 & {} \\
1 & {} & 2 \\
2 & {} & 1 \\
1 & {} & {}
\end{matrix}$, \quad
$\begin{matrix}
{} & 2 & {} \\
2 & {} & 1 \\
1 & {} & 1
\end{matrix}$;
\qquad
Injectives: $\begin{matrix}
{} & {} & {} & {} & 1 \\
2 & {} & 2 & {} & 1 \\
1 & {} & {} & 2 & {} \\
{} & 1 & {} & {} & {}
\end{matrix}$, \quad
$\begin{matrix}
{} & {} & 1 \\
2 & {} & 1 \\
{} & 2 & {}
\end{matrix}$.
\end{center}

The module $\begin{matrix} 2 \\ 1 \end{matrix}$ is infinitely deloopable, since it is the direct summand of its own syzygy. Therefore, we omit the summand $\begin{matrix} 2 \\ 1 \end{matrix}$ when calculating the delooping level of $S_2$. We know $\dell S_2 > 0$ since $S_2$ is not a submodule of a projective module. We find $\Omega S_2=\begin{matrix} 2 \\ 1 \end{matrix}\oplus \begin{matrix} 1 \\ 1 \end{matrix}$ and will show that $\begin{matrix} 1 \\ 1 \end{matrix}$ is not 2-deloopable. It suffices to show $\begin{matrix} 1 \\ 1 \end{matrix}$ is not a direct summand of $\Omega^2\agemo^2\Omega S_2 = \Omega^2\agemo^2\left( \begin{matrix} 2 \\ 1 \end{matrix} \oplus \begin{matrix} 1 \\ 1 \end{matrix}\right)$. Using the notation that modules in $\mod\Lambda^{\op}$ are underlined, we can calculate the following 
\begin{itemize}
\item $\Omega^2 \Tr \left( \begin{matrix} 2 \\ 1 \end{matrix} \right) = \underline{S_2} \oplus \underline{\begin{matrix} 1 \\ 1 \end{matrix}}$,
\item $\Omega^2\Tr \underline{S_2} = \begin{matrix} 2 \\ 1 \end{matrix} \oplus S_1$,
\item $\Omega^2\Tr \left(\underline{\begin{matrix} 1 \\ 1 \end{matrix}}\right) = \begin{matrix} 2 \\ 1 \end{matrix} \oplus \begin{matrix} 2 \\ 1 \end{matrix}$,
\item $\Omega^2\Tr \left( \begin{matrix} 1 \\ 1 \end{matrix} \right) = \underline{\begin{matrix} 2 \\ 1 \end{matrix}} \oplus \underline{\begin{matrix} 2 \\ 1 \end{matrix}}$,
\item $\Omega^2\Tr\left( \underline{\begin{matrix} 2 \\ 1 \end{matrix}} \right) = \begin{matrix} 2 \\ 1 \end{matrix} \oplus S_1$.
\end{itemize}
Therefore, $\Omega^2\agemo^2\Omega S_2 = \left(\begin{matrix} 2 \\ 1 \end{matrix}\right)^5 \oplus (S_1)^3$. Since it does not contain $\begin{matrix} 1 \\ 1 \end{matrix}$ as a summand, $\dell S_2 > 1$.

Since $\Omega^2 S_2 = \left( \begin{matrix} 2 \\ 1 \end{matrix} \right)^3 \oplus S_1$, the module $\begin{matrix} 2 \\ 1 \end{matrix}$ is infinitely deloopable, and $S_1$ is a direct summand of $\Omega^3 \left(\begin{matrix} 2 \\ 1 \end{matrix}\right)$, we find $\dell S_2 = \dell\Lambda = 2$.

On the other hand, $S_2$ is a submodule of $\begin{matrix} {} & 2 & {} \\ 2 & {} & 1 \end{matrix}$ whose syzygy is $(S_1)^2$ and infinitely deloopable, so
\[
\subddell S_2=\subddell\Lambda=1=\Findim\Lambda^{\op}<\dell\Lambda=2.
\]
\endgroup
\end{example}

In a similar vein, we ask if there are any interesting properties for quotient modules $M$ of some module $C_0$ of finite delooping level. It turns out that it is not sufficient to only consider $C_0$ and the quotient $M$. We need to consider the kernel of the projection from $C_0$ to $M$ and moreover all exact sequences that end in $M$. This motivates the definition of the derived delooping level.

\begin{definition}[\cite{guo2025derived}]
\label{def:ddell}
Let $M$ be a $\Lambda$-module.
\begin{itemize}
\item The \textbf{$k$-delooping level} of $M$ is
\begin{align*}
\kdell M = \inf\{n\in\N\mid \Omega^n M \dirsum \Omega^{n+k} N \text{ for some $N$} \}
\end{align*}
\item The \textbf{derived delooping level} of $M$ is
\begin{align*}
\ddell M = \inf \{m\in\N \mid & \,\exists n\leq m \text{ and an exact sequence in $\mod\Lambda$ of the form} \\
& \,\, 0 \to C_n \to C_{n-1} \to \cdots \to C_1 \to C_0 \to M \to 0, \\
& \text{ where $(i+1)$-$\dell C_i\leq m-i$, } i=0,1,\dots,n  \},
\end{align*}
\item The \textbf{derived delooping level} of $\Lambda$ is 
\[
\ddell \Lambda = \sup \{\ddell S \mid S \text{ is a simple $\Lambda$-module} \}.
\]
\end{itemize}
\end{definition}

\begin{theorem}
For any algebra $\Lambda$,
\[
\Findim\Lambda^{\op} \leq \ddell\Lambda \leq \dell\Lambda.
\]
\end{theorem}

The derived delooping level is a better upper bound compared to the delooping level in the same Example \ref{ex:subddell}.

\begin{example}[Example \ref{ex:subddell} Revisited]
\label{ex:ddell}
Let $\Lambda$ be the same as in Example \ref{ex:subddell}. The short exact sequence
\[
0 \to S_1 \to \begin{matrix} 2 \\ 1 \end{matrix} \to S_2 \to 0
\]
shows $\ddell S_2 = 1$. Therefore in this case,
\[
1 = \Findim\Lambda^{\op} = \ddell\Lambda = \subddell\Lambda < \dell\Lambda = 2.
\]
\end{example}

Note that if there exists a counterexample $\Lambda$ of the finitistic dimension conjecture, the derived delooping level of some simple $\Lambda^{\op}$-module $S$ must be infinite, so by Definition \ref{def:ddell}, there must be a large number of right $\Lambda^{\op}$-modules with infinite $k$-delooping level for some $k$ so that $(i+1)$-$\dell C_i=\infty$ for some $i$ in \textbf{every} exact sequence of the form $0\to C_n\to C_{n-1}\to\cdots\to C_1\to C_0\to S\to 0$. Currently, the only example with $\dell\Lambda=\infty$ in \cite{kershaw2023} features a local algebra, and it has one module with infinite delooping level. Designing algebras which have more modules of infinite delooping level or $k$-delooping levels is crucial to advancing the progress or finding counterexamples of the finitistic dimension conjecture.

\begin{question}
Can we design a finite dimensional algebra that has enough modules with infinite $k$-delooping level to make $\ddell S=\infty$ for some simple module $S$?
\end{question}

Moreover, the property of having finite derived delooping level is preserved under extensions and submodules, so the class of modules with finite derived delooping level forms a torsion-free class in $\mod\Lambda$.

\begin{theorem}[\cite{guo2025derived}]
Let $0\to A\to B\to C\to 0$ be a short exact sequence in $\Mod\Lambda$. Then
\begin{itemize}
\item If $\ddell A<\infty$ and $\ddell C<\infty$, then $\ddell B\leq\ddell A + \ddell C + 1$.
\item If $\ddell B<\infty$, then $\ddell A\leq \ddell B+1$.
\end{itemize}
\end{theorem}
Another direction of future interest is to use the derived delooping level to replicate the results for representation dimension 3 algebras and the more general Igusa-Todorov and LIT algebras mentioned in the introduction. It is still unknown whether the big finitistic dimension of these algebras are finite.

\section{Monomial Algebras}
\label{sec:Monomial algebras}

We discuss some preliminary results that we will need for the rest of the paper and use them to first provide a different proof of $\dell\Lambda=\Findim\Lambda^{\op}$ for monomial algebras whose quiver is acyclic in this section. A quiver is \textbf{acyclic} if there is no oriented cycle in the quiver. We use the idea of the proof to introduce new theorems on two special classes of algebras called left and right serial quiver algebras in Section \ref{sec:left and right serial path algebras}.

\begin{definition}
\label{def:serial path algebra}
A module is \textbf{uniserial} if its submodules are totally ordered. Alternatively, a module is uniserial if and only if it has a unique composition series. A quiver algebra $\Lambda=\k Q/I$ is \textbf{left} (resp. \textbf{right}) \textbf{serial} if each indecomposable left (resp. right) projective $\Lambda$-module is uniserial.
\end{definition}

Let $\Lambda = \k Q/I$ be a monomial algebra, where $I$ is an admissible ideal generated by monomials. There are $|Q_0|$ \textbf{trivial paths} in $Q$, one for each vertex, denoted by $e_v$ for vertex $v$. There exists a canonical set of \textbf{minimal relations} for the ideal $I$ consisting of only monomials, and we use that set of minimal relations by default and call it $B_I$. All paths in $B_I$ are called \textbf{zero paths}, and note the difference between trivial and zero paths. We call $B_Q$ the set of all nonzero paths in $Q$ including the trivial paths, so $B_Q$ is a $\k$-basis of the quiver algebra $\Lambda$. When we consider the opposite algebra $\Lambda^{\op}$, its corresponding set of minimal relations and $\k$-basis are $B_{\tilde{I}}$ and $B_{Q^{\op}}$, respectively.

It is easy to understand the second and higher syzygies of modules over a monomial algebra $\Lambda$ since all second syzygies are direct sums of $p\Lambda$ where $p$ is a path of length greater than or equal to 1 in $Q$ \cite{huisgen1991predicting}. Finding the syzygy of modules of the form $p\Lambda$ is also straightforward. The process amounts to finding nonzero paths to concatenate with $p$ to become a minimal relation in $I$. This has been pointed out in different languages in several papers including the original proof that $\Findim\Lambda<\infty$ for monomial algebras $\Lambda$ \cite{GKK1991} and in later works such as the proof that $\dell\Lambda=\Findim\Lambda^{\op}$ for truncated path algebras $\Lambda$ where the authors name the similar concept ``right complementary'' path \cite{barrios2024delooping} as opposed to ``right completion'' in our next Definition \ref{def:right completion}.

\begin{definition}
\label{def:right completion}
Let $Q=(Q_0, Q_1, s, t)$ be a quiver and $\Lambda=\k Q/I$ be a monomial algebra. For each nonzero path $\alpha$ in $Q$, we say $\beta$ is a \textbf{right completion} of $\alpha$ if $\alpha\beta\in I$, $\beta\notin I$, and for all factorizations $\beta=\beta_1\beta_2$ with $\beta_2$ nontrivial, $\alpha\beta_1\notin I$.
\end{definition}

Alternatively, Definition \ref{def:right completion} says $\beta$ is a right completion of $\alpha$ if $\beta$ is a minimal nonzero path to make $\alpha\beta$ zero. Note that there may be more than one right completions for a path. We will consider sequences of right completions $\alpha_0\alpha_1\cdots\alpha_n$ where $\alpha_{i+1}$ is a right completion of $\alpha_i$ for $i=0,1,\dots,n-1$. We say a sequence $\alpha_0\alpha_1\cdots\alpha_n$ of right completions is \textbf{right maximal} if $\alpha_n$ has no right completion, and is \textbf{left maximal} if $\alpha_0$ is not a right completion of any nonzero path. If a sequence of right completions is both left and right maximal, we say it is \textbf{maximal}. The \textbf{length} of a sequence of right completions is the number of right completions in the sequence, so the sequence $\alpha_0\alpha_1\cdots\alpha_n$ has length $n$. Note that the length is not necessarily the number of minimal relations in the sequence since there may be multiple ways to factor the whole path $\alpha_0\alpha_1\cdots\alpha_n$ into right completions. By definition, maximal sequences of right completions have finite length, and for modules with infinite projective dimension, we say they are associated with sequences of right completions of infinite length. We prove two important lemmas related to right completions.

\begin{lemma}
\label{lem:syzygy of principal ideals}
Suppose $\Lambda=\k Q/I$ is a monomial algebra. Let $\mu, \nu$ be nonzero paths of length at least 1. Then the right $\Lambda$-module $\nu\Lambda$ is a direct summand of $\Omega (\mu\Lambda)$ if and only if $\nu$ is a right completion of $\mu$.
\end{lemma}

\begin{proof}
Consider the short exact sequence
\[
0 \to \Omega(\mu\Lambda) \to P_{t(\mu)} \to \mu\Lambda \to 0.
\]

As $\k$-vector spaces, $\mu\Lambda$ has the basis $B_{\mu\Lambda} = \{\omega\in B_Q \mid \mu\omega\notin I\}$, and $P_{t(\mu)}$ has the basis $B_{P_{t(\mu)}} = \{\omega\in B_Q \mid s(\omega) = t(\mu) \}$. It is clear that $P_{t(\mu)}$ maps onto $\mu\Lambda$, and the $\k$-basis of the kernel $\Omega(\mu\Lambda)$ is
\begin{equation}
\label{eq:basis of syzygy of right ideal}
B_{\Omega(\mu\Lambda)} = \{ \nu\in B_Q \mid s(\nu)=t(\mu), \mu\nu\in I \}.
\end{equation}

From \eqref{eq:basis of syzygy of right ideal}, we note that the generators of $\Omega(\mu\Lambda)$ as a $\Lambda$-module are those $\omega r\in B_{P_{t(\mu)}}$ such that $\mu\omega r\in I$ but $\mu\omega\notin I$ for some path $\omega$ and arrow $r$. Therefore, we can rewrite
\begin{equation}
\label{eq:final basis of syzygy of right ideal}B_{\Omega(\mu\Lambda)} = \{\omega r x\in B_{P_{t(\mu)}} \mid x \text{ is a path}, \mu\omega\notin I, \mu\omega r\in I \text{ for some arrow $r$} \},
\end{equation}
so as a $\Lambda$-module $\Omega(\mu\Lambda)$ is exactly the direct sum of all $(\omega r)\Lambda$ where $\omega r$ is a right completion of $\mu$.
\end{proof}

In Lemma \ref{lem:syzygy of principal ideals}, if $\mu\Lambda$ is already projective, then by definition $\mu$ has no right completions, so indeed $\nu\Lambda$ is zero. On the other hand, if $\mu$ has no right completion, then $\mu\Lambda$ is projective. Since $\dell\Lambda$ bounds the finitistic dimension of the opposite algebra, we need to study how a sequence of right completions and its reverse are related to each other. This is discussed in Lemma \ref{lem:existence of module with pd equal dell}. To avoid confusion about which algebra we are working over, we default to using right finitistic dimensions $\findim\Lambda$ and $\Findim\Lambda$, where paths are denoted with lowercase letters in $Q$, and left delooping level $\dell\Lambda^{\op}$, where every path has a tilde above the letter in $Q^{\op}$.

Suppose we have a monomial algebra $\Lambda=\k Q/I$. Given a sequence of right completions in $Q^{\op}$ with $n\geq 1$
\begin{equation}
\label{eq:generic sequence of right completions}
\tilde{\mu}=\widetilde{p_0}\widetilde{p_1}\cdots \widetilde{p_n},
\end{equation}
there is a unique way to factor each path $\widetilde{p_i}$ into $\widetilde{x_i}\widetilde{y_i}$ such that $\widetilde{y_i}\widetilde{x_{i+1}}\widetilde{y_{i+1}}$ for $i=0,1,\dots,n-1$ and $\widetilde{y_{n-1}}\widetilde{p_n}$ are minimal relations in $\tilde{I}$. The paths $\widetilde{y_i}$ are never trivial, and $\widetilde{x_i}$ can be the trivial path. We rewrite the factored version of \eqref{eq:generic sequence of right completions} as
\begin{equation}
\label{eq:generic sequence of right completions factored}
\widetilde{x_0}\widetilde{y_0}\widetilde{x_1}\widetilde{y_1}\cdots \widetilde{x_{n-1}}\widetilde{y_{n-1}} \widetilde{p_n}.
\end{equation}
The reverse of \eqref{eq:generic sequence of right completions factored} in $Q$ is
\begin{equation}
\label{eq:generic sequence of right completions reversed}
p_n y_{n-1}x_{n-1}\cdots y_1x_1y_0x_0.
\end{equation}
Note that even if the sequence in \eqref{eq:generic sequence of right completions} is maximal, the reversed sequence \eqref{eq:generic sequence of right completions reversed} may not be left maximal or right maximal due to the possibility of multiple arrows going in and out of each vertex. However, we know for sure that the $\Lambda$-module $M=\Lambda/p_n\Lambda$ has projective dimension at least $n+1$ by identifying all possible minimal relations that can occur in \eqref{eq:generic sequence of right completions reversed}.

\begin{lemma}
\label{lem:existence of module with pd equal dell}
Using the notation in \eqref{eq:generic sequence of right completions}, \eqref{eq:generic sequence of right completions factored}, and \eqref{eq:generic sequence of right completions reversed}, there are at least $n$ minimal relations in \eqref{eq:generic sequence of right completions reversed}, and $\pd\!_{\Lambda} (\Lambda/p_n\Lambda) \geq n+1$.
\end{lemma}

\begin{proof}
It is clear that the numbers of minimal relations in \eqref{eq:generic sequence of right completions factored} and \eqref{eq:generic sequence of right completions reversed} are the same and that there are at least $n$ minimal relations in \eqref{eq:generic sequence of right completions reversed}, which are $p_ny_{n-1}$, $y_{n-1}x_{n-1}y_{n-2}, \dots, y_1x_1y_0$. There may be more minimal relations in \eqref{eq:generic sequence of right completions factored}. If there are additional minimal relations in \eqref{eq:generic sequence of right completions factored}, they cannot start with any arrow in $\widetilde{x_i}$ for $i=0,1,\dots,n-1$. Without the loss of generality, suppose $\widetilde{x_0}$ is nontrivial and for a contradiction that there is a minimal relation starting with some arrow of $\widetilde{x_0}$. The terminal arrow of the relation cannot be in or before $\widetilde{y_0}$ since $\widetilde{x_0}\widetilde{y_0}$ is nonzero. The terminal arrow cannot be the last arrow of or after $\widetilde{y_1}$ since $\widetilde{y_0}\widetilde{x_1}\widetilde{y_1}$ is a minimal relation. Lastly, if the terminal arrow is in $\widetilde{x_1}$ or $\widetilde{y_1}$ except for the last arrow, then $\widetilde{x_1}\widetilde{y_1}$ is not a right completion of $\widetilde{x_0}\widetilde{y_0}$ and should instead be shorter. Therefore, there is no minimal relation in \eqref{eq:generic sequence of right completions factored} starting with any arrow in $\widetilde{x_i}$. On the other hand, there may be relations whose starting arrow is an arrow in $\widetilde{y_i}$ and terminal arrow before the end of $\widetilde{y_{i+2}}$, but this does not affect the results in the rest of the proof.

Let $N\in\mod\Lambda$ be the non-projective summand of $\Lambda/p_n\Lambda$. If $n\leq 3$, then by applying Lemma \ref{lem:syzygy of principal ideals} with the minimal relations in $B_I$, we get $\Omega_{\Lambda} N = p_n\Lambda$, $ y_{n-1}x_{n-1}\Lambda \dirsum \Omega_{\Lambda}^2 N$, $y_{n-2}\Lambda \dirsum \Omega_{\Lambda}^3 N$, and $x_{n-2}y_{n-3}\Lambda \dirsum \Omega_{\Lambda}^4 N$ whenever applicable. For $n\geq 4$, since we showed in the previous paragraph that all other minimal relations in \eqref{eq:generic sequence of right completions reversed} must have their starting and terminal arrows in $y_{i+2}$ and $y_i$ for some $i$, we can describe the higher syzygies of $N$ as $\Omega^j_{\Lambda} N = x_{n-j+2}y_{n-j+1}\Lambda$ for $j=5,\dots,n+1$. Note that $\Omega^j_{\Lambda} N$ is not projective for $j<n+1$ since their generator is a path that has a right completion, so $\pd\!_{\Lambda} N \geq n+1$.
\end{proof}

It is already shown that if $\Lambda$ is a Nakayama algebra \cite{ringel2021finitistic, sen2021delooping} or a monomial algebra whose underlying quiver is acyclic (for example in \cite[Proposition 2.3]{guo2024symmetry} since the algebra has finite global dimension), then
\begin{equation}
\label{eq:all invariants equal}
\ddell\Lambda=\ddell\Lambda^{\op}=\dell\Lambda=\dell\Lambda^{\op}=\findim\Lambda=\findim\Lambda^{\op}=\Findim\Lambda=\Findim\Lambda^{\op}.
\end{equation}
However, we present a different proof for the case of monomial algebras of acyclic quivers that can be generalized to more cases in the next section by demonstrating how the delooping level of a simple module in one algebra corresponds to the projective dimension of modules generated by paths in the opposite algebra.

\begin{proposition}
\label{prop:tree dell findim}
If $Q$ is an acyclic quiver and $\Lambda=\k Q/I$ is a monomial algebra, then 
\[
\gldim\Lambda=\Findim\Lambda=\findim\Lambda = \dell\Lambda^{\op}<\infty.
\]
\end{proposition}

\begin{proof}
Since $Q$ is an acyclic quiver, $\gldim\Lambda<\infty$, so the statement is automatically true, but we present a proof that can be generalized to other classes of monomial algebras. We will prove $\dell\Lambda^{\op}\leq \findim\Lambda$ since $\findim\Lambda\leq \Findim\Lambda\leq \dell\Lambda^{\op}$. We can assume $\dell\Lambda^{\op}\geq 1$ since it is known that $\Findim\Lambda=0$ if and only if $\dell\Lambda^{\op}=0$, for example in \cite{krause2023symmetry}.

Consider every simple module $S=S_{\tilde{v}}\in\mod\Lambda^{\op}$ supported on some vertex $\tilde{v}$. There exist sequences of right completions corresponding to $S_{\tilde{v}}$ all of the form
\begin{equation}
\label{eq:tree dell findim}
\widetilde{p_0}\widetilde{p_1}\widetilde{p_2}\cdots \widetilde{p_{n-1}}\widetilde{p_n},
\end{equation}
where $\widetilde{p_0}$ is an arrow starting at $\tilde{v}$, $\widetilde{p_{i+1}}$ is a right completion of $\widetilde{p_i}$ for $i=0,1,\dots,n-1$, and $\widetilde{p_n}\Lambda^{\op}$ is projective so that \eqref{eq:tree dell findim} is right maximal. Iterating over all simple $\Lambda^{\op}$-modules and all such sequences of right completions corresponding to them, we pick any longest sequence and write it in the form \eqref{eq:tree dell findim}. This implies $\dell\Lambda^{\op}\leq n+1$. Note that this does not imply $\dell\Lambda^{\op}=n+1$ since earlier syzygies of $S_{\tilde{v}}$ could be more deloopable making its delooping level smaller, but we show later this does not happen and the equality must hold.

To describe the minimal relations in $B_{\tilde{I}}$ that occur in \eqref{eq:tree dell findim}, we factor each path $\widetilde{p_i}$ into a concatenation of two paths as in \eqref{eq:generic sequence of right completions factored}. The sequence \eqref{eq:tree dell findim} becomes
\begin{equation}
\label{eq:tree dell findim factored}
\widetilde{y_0}\widetilde{x_1}\widetilde{y_1}\cdots \widetilde{x_{n-1}}\widetilde{y_{n-1}} \widetilde{p_n},
\end{equation}
where $\widetilde{x_0}$ is always trivial and is therefore omitted because the syzygy of $S_{\tilde{v}}$ is the direct sum of all $\widetilde{y_0}\Lambda^{\op}$ with $\widetilde{y_0}$ an arrow starting from $\tilde{v}$.

We reverse \eqref{eq:tree dell findim factored} to obtain a sequence of right completions in $Q$
\begin{equation}
\label{eq:tree dell findim reversed}
p_n y_{n-1}x_{n-1}\cdots y_1x_1y_0.
\end{equation}
Let $M$ be the non-projective summand of $\Lambda/p_n\Lambda$. By Lemma \ref{lem:existence of module with pd equal dell}, we know $\pd\!_{\Lambda} M\geq n+1$. Since $\pd\!_{\Lambda} M$ is finite, we get
\[
\dell\Lambda^{\op} = \dell\!_{\Lambda^{\op}} S_{\tilde{v}} \leq n+1 \leq \pd\!_{\Lambda} M \leq \findim\Lambda,
\]
completing the proof.
\end{proof}

We reiterate that the result in Proposition \ref{prop:tree dell findim} can be observed by using the property of the delooping level, but the proof above demonstrates intuitively why the equality $\dell\Lambda^{\op}=\Findim\Lambda$ holds through reversing maximal sequences of right completions. The key condition in the proof is that all sequences of right completions there have finite length since every module has finite projective dimension. We extend the idea of this proof to left and right serial algebras in the next section. Proposition \ref{prop:tree dell findim} also recovers the result that $\findim\Lambda=\Findim\Lambda$ if $\Lambda$ is a monomial algebra whose quiver is acyclic and that the finitistic dimension can be achieved by the quotient of a projective by a principal ideal generated by some path in the quiver. 

\section{Left and Right Serial Quiver Algebras}
\label{sec:left and right serial path algebras}

We begin the section by describing the quivers of left and right serial algebras.

\begin{lemma}
Suppose every vertex in a connected quiver has outdegree at most one. If the quiver has a cycle, then the cycle must be an oriented cycle. Moreover, the quiver has at most one cycle. The same is true if every vertex in a connected quiver has indegree at most one.
\end{lemma}

\begin{proof}
For both cases, it is clear that all cycles in the quiver must have straight orientation to keep all indegrees or outdegrees at most one. If the quiver has two cycles or more, then the vertices that connect the cycles will have indegree or outdegree larger than one.
\end{proof}

Therefore, the quiver of a left serial quiver algebra is either a tree or a cycle with straight orientation in which each vertex can have additional outgoing arrows that are part of a tree, and the trees are not connected to each other by any arrow. Similarly, the quiver of a right serial quiver algebra is the same except there can be additional incoming arrows out of each vertex in the cycle. If there is an oriented cycle in the quiver of a left or right serial quiver algebra, we call the vertices in the cycle the \textbf{cyclic part} of the quiver. The other vertices are called the \textbf{tree part} of the quiver.

We extend the validity of $\dell\Lambda^{\op}=\ddell\Lambda^{\op}=\Findim\Lambda$ to right serial quiver algebras $\Lambda$ and show that the two upper bounds are not necessarily tight if $\Lambda$ is a left serial algebra. We prove the result about right serial quiver algebras in Theorem \ref{thm:right serial} and recover the result that the right finitistic dimensions $\findim\Lambda=\Findim\Lambda$ are equal if $\Lambda$ is right serial \cite{huisgen1992syzygies}. Note that the convention for path concatenation in \cite{huisgen1992syzygies} is the opposite to ours, so their results for left serial algebras are for right serial algebras in our context. The author in \cite{huisgen1992syzygies} also provides a method to calculate the finitistic dimensions through uniserial ideals and points out that the calculation only depends on the quiver and the relations. Our conclusion agrees with these statements and at the same time shows the delooping level and the derived delooping level serve as another tractable way to calculate the finitistic dimensions.

\begin{theorem}
\label{thm:right serial}
Let $\Lambda=\k Q/I$ be a right serial quiver algebra. Then the right little and big finitistic dimensions of $\Lambda$ are equal to the left delooping level and derived delooping level of $\Lambda$, \textit{i.e.},
\[
\findim\Lambda = \Findim\Lambda = \dell\Lambda^{\op} = \ddell\Lambda^{\op} < \infty.
\]
\end{theorem}

\begin{proof}
Since right serial quiver algebras are monomial, the four quantities are all finite. Also note that each vertex in $Q$ (resp. $Q^{\op}$) has at most one outgoing (resp. incoming) arrow. As in the proof of Proposition \ref{prop:tree dell findim}, it suffices to show $\dell\Lambda^{\op}\leq\findim\Lambda$. 

Let $S=S_{\tilde{v}}$ be a simple module in $\mod\Lambda^{\op}$ such that $\dell\!_{\Lambda^{\op}} S = \dell\Lambda^{\op}$. We know $\findim\Lambda=\Findim\Lambda=0$ if and only if $\dell\Lambda^{\op}=0$, so the theorem is true when $\dell\Lambda^{\op}$ is 0. The theorem is also true when $\dell\Lambda^{\op}=1$ since $\Findim\Lambda$ can only be 0 or 1 in that case, and $\Findim\Lambda=0$ implies $\dell\Lambda^{\op}=0$. We will prove the statement for $\dell\Lambda^{\op}\geq 3$ (corresponding to $n\geq 2$ in \eqref{eq:right serial right completion factored}), so the remaining $\dell\Lambda^{\op}=2$ case will follow, which we also explain in Remark \ref{rmk:dell=2 case}. Consider any arrow $\widetilde{y_0}$ starting at $\tilde{v}$ so that there exists a sequence of right completions
\begin{equation}
\label{eq:right serial right completion factored}
\widetilde{y_0}\widetilde{x_1}\widetilde{y_1}\widetilde{x_2}\widetilde{y_2} \cdots \widetilde{x_{n-1}}\widetilde{y_{n-1}}\widetilde{p_n}.
\end{equation}
The sequence \eqref{eq:right serial right completion factored} is factored in the same way as \eqref{eq:tree dell findim factored} in the proof of Proposition \ref{prop:tree dell findim}, where $\widetilde{y_{n-1}}\widetilde{p_n}$ and $\widetilde{y_i}\widetilde{x_{i+1}}\widetilde{y_{i+1}}$ for $i=0,1,\dots,n-2$ are minimal relations in $B_{\tilde{I}}$. Since $\dell\Lambda^{\op} = \dell_{\Lambda^{\op}} S = n+1 \geq 3$, we may choose \eqref{eq:right serial right completion factored} to correspond to the information that there exists a summand of the form $\widetilde{x_{n-1}}\widetilde{y_{n-1}}\Lambda^{\op}$ of $\Omega_{\Lambda^{\op}}^n S$ that is not $(n+1)$-deloopable and after we take another syzygy, the summand $\widetilde{p_n}\Lambda^{\op}$ of $\Omega_{\Lambda^{\op}}^{n+1} S$ is at least $(n+2)$-deloopable. We also know $\widetilde{y_0}\Lambda^{\op}\dirsum \Omega_{\Lambda^{\op}} S$ and $\widetilde{x_i}\widetilde{y_i}\Lambda^{\op}\dirsum \Omega_{\Lambda^{\op}}^{i+1} S$ for $1\leq i\leq n-1$.

Consider the reverse of \eqref{eq:right serial right completion factored} in $Q$ written as
\begin{equation}
\label{eq:right serial right completion reversed}
p_n y_{n-1}x_{n-1}\cdots y_1x_1y_0.
\end{equation}
We show that $x_1y_0\Lambda$ is always projective in $\mod\Lambda$. First, if $n=2$, $\dell\Lambda^{\op}=3$, but $x_1y_0\Lambda$ is not projective, then there exists a right completion $p$ for $x_1y_0$ in $Q$. We factor $x_1=z_1z_2$ such that $z_2y_0p$ is a minimal relation in $B_I$. Then we have a sequence of right completions $p_2y_1z_1z_2y_0p$ in $Q$, and its reverse in $Q^{\op}$ is
\begin{equation}
\label{eq:right serial dell=3 sequence}
\widetilde{p}\widetilde{y_0}\widetilde{z_2}\widetilde{z_1}\widetilde{y_1}\widetilde{p_2},
\end{equation}
where the minimal relations are $\widetilde{p}\widetilde{y_0}\widetilde{z_2}$, $\widetilde{y_0}\widetilde{z_2}\widetilde{z_1}\widetilde{y_1}$, and $\widetilde{y_1}\widetilde{p_2}$. So, we find that $\widetilde{z_1}\widetilde{y_1}\Lambda^{\op}$ is a summand of $\Omega^3_{\Lambda^{\op}} (\Lambda^{\op}/\widetilde{p}\Lambda^{op})$. On the other hand, $\widetilde{z_2}\widetilde{z_1}\widetilde{y_1}\Lambda^{\op}$ is a summand of $\Omega^2_{\Lambda^{\op}} S_{\tilde{v}}$. We showed in the proof of Lemma \ref{lem:existence of module with pd equal dell} that there is no minimal relation in \eqref{eq:right serial dell=3 sequence} starting from any arrow in $\widetilde{z_2}\widetilde{z_1}$, so the summands of $\widetilde{z_2}\widetilde{z_1}\widetilde{y_1}\Lambda^{\op}$ and $\widetilde{z_1}\widetilde{y_1}\Lambda^{\op}$ that are supported in the vertices in $\widetilde{p_2}$ are equal. We can apply this argument to any such sequence in the form \eqref{eq:right serial dell=3 sequence} where $\widetilde{y_0}$ is an arrow going out of $\tilde{v}$. This shows $\dell\!_{\Lambda^{\op}} S_{\tilde{v}} \leq 2$, contradicting the assumption.

Now we continue to the case $n\geq 3$ and $\dell\Lambda^{\op}\geq 4$. Suppose for a contradiction that $x_1y_0\Lambda$ is not projective. Then in the same way there exist a nonzero path $p$ and a factorization $x_1=z_1z_2$ such that $z_2y_0p$ is a minimal relation in $B_I$. Reversing to $Q^{\op}$, we get the sequence of right completions
\begin{equation}
\label{eq:right serial right completion factored additional term}
\widetilde{p}\widetilde{y_0}\widetilde{z_2}\widetilde{z_1}\widetilde{y_1}\widetilde{x_2}\widetilde{y_2} \cdots \widetilde{x_{n-1}}\widetilde{y_{n-1}}\widetilde{p_n}.
\end{equation}
Let $M$ be the non-projective summand of $\Lambda^{\op}/\tilde{p}\Lambda^{\op}\in\mod\Lambda^{\op}$. We get that the summand $\widetilde{x_2}\widetilde{y_2}\Lambda^{\op}$ of $\Omega_{\Lambda^{\op}}^3 S$ is a summand of the fourth syzygy of $M$. This also contradicts the assumption that $\dell\!_{\Lambda^{\op}} S\geq 4$. Therefore, $x_1y_0\Lambda$ is always projective in $\mod\Lambda$.

Let $N\in\mod\Lambda$ be the non-projective summand of $\Lambda/p_n\Lambda$. Note that since $\Lambda$ is right serial, there is no other sequence of right completions starting with any arrow in \eqref{eq:right serial right completion reversed} except for the sequence \eqref{eq:right serial right completion reversed} itself. From the relations in \eqref{eq:right serial right completion reversed}, we find that $\Omega_{\Lambda}^n N = x_2y_1\Lambda$ is not projective, but $\Omega_{\Lambda}^{n+1} N= x_1y_0\Lambda$ is projective. Thus, we complete the proof because
\[
\findim\Lambda \geq \pd\!_{\Lambda} N = n+1 = \dell\Lambda^{\op}\geq\ddell\Lambda^{\op}\geq\Findim\Lambda\geq\findim\Lambda.
\]
\end{proof}

\begin{remark}
\label{rmk:dell=2 case}
Note that in the proof of Theorem \ref{thm:right serial}, we did not need to explicitly prove the case $\dell\Lambda^{\op}=2$. In that case, the sequence of right completions in $Q^{\op}$ is simply $\widetilde{y_0}\widetilde{p_1}$. Using the same argument as in the proof, we can show $y_0\Lambda$ must be projective in $\mod\Lambda$ to ensure $\dell\Lambda^{\op} > 1$. Since the argument is essentially verbatim, we did not include it in the proof, but the same connection between the delooping level and projective dimension still exists for $\dell\Lambda^{\op}=2$.
\end{remark}

The theorem immediately shows a type of modules whose projective dimension equals the finitistic dimension.

\begin{corollary}
If $\Lambda=\k Q/I$ is a right serial quiver algebra, then the right finitistic dimensions $\findim\Lambda$ and $\Findim\Lambda$ are equal. In particular, this number can be achieved by the projective dimension of a finitely generated module of the form $\Lambda/p\Lambda$ for some path $p$.
\end{corollary}

The theorem also recovers the result that $\dell\Lambda=\Findim\Lambda^{\op}=\findim\Lambda^{\op}$ and the finitistic dimension is left-right symmetric for Nakayama algebras $\Lambda$.

\begin{corollary}
If $\Lambda$ is a Nakayama algebra, then
\[
\ddell\Lambda=\ddell\Lambda^{\op}=\dell\Lambda=\dell\Lambda^{\op}=\findim\Lambda=\findim\Lambda^{\op}=\Findim\Lambda=\Findim\Lambda^{\op}.
\]
\end{corollary}

\begin{proof}
Since $\Lambda$ is right serial, we have $\findim\Lambda = \Findim\Lambda = \dell\Lambda^{\op} = \ddell\Lambda^{\op}$. Suppose the sequences that achieves the delooping level and maximum projective dimension are 
\[
\widetilde{y_0}\widetilde{x_1}\widetilde{y_1}\cdots\widetilde{x_{n-1}}\widetilde{y_{n-1}}\widetilde{p_n}
\]
in $Q^{\op}$ and
\[p_ny_{n-1}x_{n-1}y_{n-2}\cdots y_1x_1y_0
\]
in $Q$. The latter sequence shows $\pd\!_{\Lambda} (\Lambda/p_n\Lambda) = n+1$ and $x_1y_0\Lambda$ is projective. Factoring the path $p_n$ into $y_nx_n$ such that $y_n$ is an arrow, we rewrite the sequence of right completions in $Q$ as
\begin{equation}
\label{eq:Nakayama}
y_nx_ny_{n-1}x_{n-1}y_{n-2}\cdots y_1x_1y_0.
\end{equation}

Then the simple module $S=S_{s(y_0)}$ has finite projective dimension $n+1$ since
\begin{itemize}
\item $\Omega_{\Lambda} S = y_n\Lambda$
\item $\Omega_{\Lambda}^2 S = x_ny_{n-1}\Lambda$
\item $\Omega_{\Lambda}^n S = x_2y_1\Lambda$, which is not projective
\item $\Omega_{\Lambda}^{n+1} S = x_1y_0\Lambda$, which is projective
\end{itemize}

This implies $\findim\Lambda\geq n+1=\dell\Lambda^{\op}\geq\dell\Lambda$, but since the argument is symmetric, we also have $\dell\Lambda\geq \dell\Lambda^{\op}$, completing the proof.
\end{proof}

For left serial quiver algebras $\Lambda$, $\dell\Lambda=\Findim\Lambda^{\op}$ is not necessarily true, even for representation-finite algebras. We provide a sufficient condition for when the equality holds and an example (Example \ref{ex:right serial}) for when the equality fails if the sufficient condition is not satisfied.

\begin{theorem}
\label{thm:left serial}
Let $\Lambda=\k Q/I$ be a left serial quiver algebra. If every simple $\Lambda^{\op}$-module $S$ with $\dell\!_{\Lambda^{\op}} S = \dell\Lambda^{\op}$ has its corresponding sequence of right completions entirely supported on the cyclic part or the tree part of the quiver, then the right little and big finitistic dimensions of $\Lambda$ are equal to the left delooping level and derived delooping level of $\Lambda$, \textit{i.e.},
\[
\findim\Lambda = \Findim\Lambda = \dell\Lambda^{\op} = \ddell\Lambda^{\op} < \infty.
\]
\end{theorem}

\begin{proof}
The proof is similar to that of Theorem \ref{thm:right serial}. It suffices to prove $\dell\Lambda^{\op}\leq\findim\Lambda$. Let $S=S_{\tilde{v}}$ be a simple module in $\mod\Lambda^{\op}$ such that $\dell\!_{\Lambda^{\op}} S = \dell\Lambda^{\op} = n+1$, which corresponds to the sequence of right completions
\begin{equation}
\label{eq:left serial right completion factored}
\widetilde{y_0}\widetilde{x_1}\widetilde{y_1}\widetilde{x_2}\widetilde{y_2} \cdots \widetilde{x_{n-1}}\widetilde{y_{n-1}}\widetilde{p_n},
\end{equation}
where this sequence is unique since each vertex in $Q^{\op}$ has at most one outgoing arrow. The cases when $\dell\Lambda^{\op}$ is zero, one, or two are treated the same way as in the proof of Theorem \ref{thm:right serial}.

For $n\geq 2$ and $\dell\Lambda^{\op}\geq 3$, the same argument shows $x_1y_0\Lambda$ is a projective $\Lambda$-module. We reverse \eqref{eq:left serial right completion factored} to get a sequence of right completions in $Q$
\begin{equation}
\label{eq:left serial right completion reversed}
p_n y_{n-1}x_{n-1}\cdots y_1x_1y_0.
\end{equation}
Let $M=\Lambda/p_n\Lambda\in\mod\Lambda$. We know $\pd\!_{\Lambda} M\geq n+1$. Since there is no restriction on the number of arrows going out of each vertex in $Q$, $\pd\!_{\Lambda} M$ could be infinite. However, we assumed that all vertices in \eqref{eq:left serial right completion reversed} are all supported either in the cyclic part or in the tree part. In both cases, other sequences of right completions starting with $p_n$ have finite length as they can only stay in the tree part. Therefore, $n+1\leq\pd\!_{\Lambda} M<\infty$.

\end{proof}

The sufficient condition that is in the statement of Theorem \ref{thm:left serial} is clearly not a necessary condition, since the algebra can have finite global dimension without satisfying the condition. If the condition is not satisfied, $\dell\Lambda^{\op}=\Findim\Lambda$ does not hold even for representation-finite algebras. We demonstrate it in the next example and show that $\ddell\Lambda=\Findim\Lambda^{\op}$ in that example.

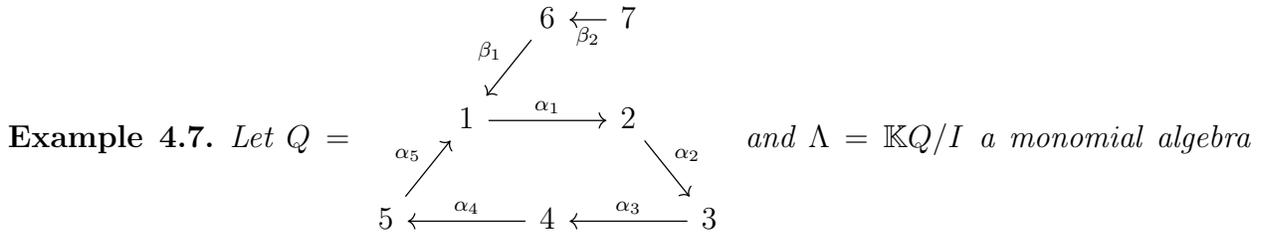
\begin{example}
\label{ex:right serial}
\begingroup
\setlength\arraycolsep{1pt}
Let $Q=
\begin{tikzcd}[ampersand replacement=\&,column sep=small]
{} \& {} \& 6 \arrow[dl, "\beta_1",swap] \& 7 \arrow[l, "\beta_2"] \\
{} \& 1 \arrow[dl,leftarrow,"\alpha_5",swap] \arrow[rr,"\alpha_1"] \& {} \& 2 \arrow[dr,"\alpha_2"] \& {} \\
5 \arrow[rr,leftarrow,"\alpha_4"] \& {} \& 4 \& {} \& 3 \arrow[ll,"\alpha_3",swap] \\
\end{tikzcd}$
and $\Lambda=\mathbb{K}Q/I$ a monomial algebra such that the indecomposable projective modules of $\Lambda$ and $\Lambda^{\op}$ are
\[
\text{$\Lambda$-modules: } \begin{matrix} 1 \\ 2 \\ 3 \end{matrix} \quad \begin{matrix} 2 \\ 3 \\ 4 \end{matrix} \quad \begin{matrix} 3 \\ 4 \\ 5 \\ 1 \end{matrix} \quad \begin{matrix} 4 \\ 5 \\ 1 \\ 2 \end{matrix} \quad \begin{matrix} 5 \\ 1 \\ 2 \\ 3 \end{matrix} \quad \begin{matrix} 6 \\ 1 \end{matrix} \quad \begin{matrix} 7 \\ 6 \end{matrix}; \qquad
\text{$\Lambda^{\op}$-modules: } \begin{matrix} {} & 1 & {} \\ 5 & {} & 6 \\ 4 & {} & {} \\ 3 & {} & {} \end{matrix} \quad \begin{matrix} 2 \\ 1 \\ 5 \\ 4 \end{matrix} \quad \begin{matrix} 3 \\ 2 \\ 1 \\ 5 \end{matrix} \quad \begin{matrix} 4 \\ 3 \\ 2 \end{matrix} \quad \begin{matrix} 5 \\ 4 \\ 3 \end{matrix} \quad \begin{matrix} 6 \\ 7 \end{matrix} \quad 7.
\]

We can show that $\Lambda$ is representation-finite, so it is straightforward to find $\Findim\Lambda=1$ and $\Findim\Lambda^{\op}=2$. Each vertex in $Q$ has at most one outgoing arrow, so $\Lambda$ is right serial and $\Lambda^{\op}$ is left serial. 

\underline{$\Lambda$ right serial:} Starting with $\dell\Lambda^{\op}$, we find that 
the simple module supported on vertex 1 in $\mod\Lambda^{\op}$ is the only simple with nonzero delooping level. In the truncated projective resolution
\[
\text{in } \mod\Lambda^{\op}: 0 \to \begin{matrix} 5 \\ 4 \\ 3 \end{matrix} \oplus \widetilde{S_6} \to \begin{matrix} {} & 1 & {} \\ 5 & {} & 6 \\ 4 & {} & {} \\ 3 & {} & {} \end{matrix} \to \widetilde{S_1} \to 0,
\]
the summand $\widetilde{S_6}$ of $\Omega_{\Lambda^{\op}} \widetilde{S_1}$ is a second syzygy of $\begin{matrix} 3 \\ 2 \end{matrix}$, so $\dell\Lambda^{\op}=1=\Findim\Lambda$. This corresponds to a trivial case in Theorem \ref{thm:right serial}.

\underline{$\Lambda^{\op}$ left serial:} Now we show $3=\dell\Lambda>\Findim\Lambda^{\op}=\ddell\Lambda=2$. For $\dell\Lambda$, $S_5$ and $S_7$ are the two simple modules with nonzero delooping level. It is clear that $\dell\!_{\Lambda} S_5 = 1$. The truncated projective resolution of $S_7$ below shows $\dell\!_{\Lambda} S_7=3$
\begin{equation}
\label{eq:left serial dell neq Findim}
\begin{tikzcd}[ampersand replacement=\&, row sep=tiny, column sep=small]
0 \ar[r] \& \Omega_{\Lambda}^3 S_7 = \begin{matrix} 2 \\ 3 \end{matrix} \ar[r] \& \begin{matrix} 1 \\ 2 \\ 3 \end{matrix} \ar[rr] \ar[rd] \& {} \& \begin{matrix} 6 \\ 1 \end{matrix} \ar[rr] \ar[rd] \& {} \& \begin{matrix} 7 \\ 6 \end{matrix} \ar[r] \& S_7 \ar[r] \& 0, \\
{} \& {} \& {} \& \Omega_{\Lambda}^2 S_7 = S_1 \ar[ur,hookrightarrow] \& {} \& \Omega_{\Lambda} S_7 = S_6 \ar[ur,hookrightarrow] \& {} \& {} \& {}
\end{tikzcd}
\end{equation}
since $\Omega_{\Lambda}^3 S_7$ is infinitely deloopable
\begin{itemize}
\item $\Omega_{\Lambda} S_7 = S_6$ is not a direct summand of $\Omega^2\agemo^2 S_6 = 0$
\item $\Omega_{\Lambda}^2 S_7 = S_1$ is not a direct summand of $\Omega^3\agemo^3 S_1 = \begin{matrix} 5 \\ 1 \end{matrix}$
\item $\Omega_{\Lambda}^3 S_7 = \begin{matrix} 2 \\ 3\end{matrix} = \Omega_{\Lambda}^4 \left( \begin{matrix} 4 \\ 5 \\ 1 \end{matrix} \right)$.
\end{itemize}

The sequence of right completions corresponding to \eqref{eq:left serial dell neq Findim} is $\beta_2\beta_1\alpha_1$ with $\beta_2\beta_1$ and $\beta_1\alpha_1$ being the minimal relations. Note that the vertices in the sequence are in both the cyclic and tree part of the quiver, so the condition in Theorem \ref{thm:left serial} is not satisfied. The candidate $\Lambda^{\op}$-module considered in Theorem \ref{thm:left serial} is $\Lambda^{\op}/\widetilde{\alpha_1}\Lambda^{\op}=\widetilde{S_2}$. However, since the reversed sequence of right completions $\widetilde{\alpha_1}\widetilde{\beta_1}\widetilde{\beta_2}$ has into the cyclic part another branch $\widetilde{\alpha_1}\widetilde{\alpha_5}\widetilde{\alpha_4}\widetilde{\alpha_3}\widetilde{\alpha_2}\widetilde{\alpha_1}\widetilde{\alpha_5}\widetilde{\alpha_4}\widetilde{\alpha_3}\cdots$ that goes on infinitely, $\pd\!_{\Lambda^{\op}} \widetilde{S_2} = \infty$. The reverse of the subsequence $\beta_2\beta_1$ which is completely in the tree part leads to the module $\Lambda^{\op}/\widetilde{\beta_1}\Lambda^{\op}$, which has projective dimension 2 and the projective resolution
\[
0 \to S_7 \to \begin{matrix} 6 \\ 7 \end{matrix} \to \begin{matrix} {} & 1 & {} \\ 5 & {} & 6 \\ 4 & {} & {} \\ 3 & {} & {} \end{matrix} \to \begin{matrix} 1 \\ 5 \\ 4 \\ 3 \end{matrix} \to 0.
\]

However, $\ddell\Lambda=\Findim\Lambda^{\op}=2$ using the following exact sequence
\[
\begingroup
\setlength\arraycolsep{0pt}
0 \to \begin{matrix} 5 \\ 1 \end{matrix} \to \begin{matrix} 5 & & 6 \\ & 1 & \end{matrix} \to \begin{matrix} 7 \\ 6 \end{matrix} \to S_7 \to 0
\endgroup
\]
shows $\ddell\Lambda=2=\Findim\Lambda^{\op}$, since
\[
\begin{matrix} 5 \\ 1 \end{matrix} = \Omega^3 \left( \begin{matrix} 4 \\ 5 \\ 1 \end{matrix}\right) \Rightarrow 3\text{-}\dell\! \left(\begin{matrix} 5 \\ 1 \end{matrix}\right) \leq 0, \quad \Omega \left(\begingroup
\setlength\arraycolsep{0pt} \begin{matrix} 5 & & 6 \\ & 1 & \end{matrix} \endgroup \right) = \begin{matrix} 1 \\ 2 \\ 3 \end{matrix} \Rightarrow 2\text{-}\dell\! \left(\begingroup
\setlength\arraycolsep{0pt} \begin{matrix} 5 & & 6 \\ & 1 & \end{matrix} \endgroup \right) \leq 1, \quad 1\text{-}\dell\!\left(\begin{matrix} 7 \\ 6 \end{matrix}\right)\leq 2.
\]
\endgroup
\end{example}

We end the paper with some open questions. In the monomial example where $\dell\Lambda^{\op}-\Findim\Lambda$ can be made arbitrarily large \cite{barrios2024delooping}, the corresponding quiver has five arrows going out of the ``source'' vertex. In our example where their difference is 1, each vertex has at most two outgoing arrows. We ask whether the relationship among $\dell\Lambda^{\op}$, $\ddell\Lambda^{\op}$, and $\Findim\Lambda$ can be quantified by the maximum incoming/outgoing arrows out of each vertex.

\begin{question}
Given a quiver, there are many ways to manipulate the relations to create monomial algebras with desired finitistic dimensions \cite{happel2013algebras}. Similarly, if $\Lambda$ is a monomial algebra, can we quantify the differences $\dell\Lambda^{\op}-\Findim\Lambda$ and $\ddell\Lambda^{\op}-\Findim\Lambda$ in terms of the underlying quiver of $\Lambda$ and its relations?
\end{question}

\begin{question}
For left serial algebras $\Lambda$, is there a class of examples where $\dell\Lambda^{\op}-\Findim\Lambda$ and/or $\ddell\Lambda^{\op}-\Findim\Lambda$ becomes arbitrarily large? This would simplify the example in \cite{barrios2024delooping}.
\end{question}

\bibliographystyle{plain}
\bibliography{refs}

\end{document}